\documentclass[12pt]{amsproc}

\usepackage{amsmath, amsthm, amssymb, amsfonts, mathrsfs, amscd}

\def\AA{{\mathbb A}}

\def\FF{{\mathbb F}}
\def\GG{{\mathbb G}}

\def\QQ{{\mathbb Q}}
\def\PP{{\mathbb P}}
\def\QQ{{\mathbb Q}}

\def\ZZ{{\mathbb Z}}

\def\0{{\mathbf 0}}
\def\1{{\mathbf 1}}
\def\abf{{\mathbf a}}
\def\bbf{{\mathbf b}}

\def\Acal{{\mathcal A}}

\def\Fcal{{\mathcal F}}
\def\Gcal{{\mathcal G}}

\def\Mcal{{\mathcal M}}

\def\Ocal{{\mathcal O}}

\def\Vcal{{\mathcal V}}

\def\Kbar{{\bar K}}

\def\Rat{\mathrm{Rat}}
\def\Spec{\mathrm{Spec}}

\def\PGL{\mathrm{PGL}}
\def\GL{\mathrm{GL}}

\def\Fix{\mathrm{Fix}}

\def\Res{\mathrm{Res}}

\def\uf{\mathrm{uf}}
\def\uc{\mathrm{uc}}

\def\Ratfd2{\mathrm{Rat}_{d,2}^{\uf}}
\def\Mfd2{\mathcal{M}_{d,2}^{\uf}}
\def\Rfd2{\mathcal{R}_{d,2}^{\uf}}
\def\Rcd2{\mathcal{R}_{d,2}^{\uc}}

\theoremstyle{plain}

\newtheorem{thm}{Theorem}

\newtheorem{prop}[thm]{Proposition}
\newtheorem{lem}[thm]{Lemma}

\theoremstyle{definition}
\newtheorem*{dfn}{Definition}

\newtheorem*{dfns}{Definitions}
\newtheorem{rem}{Remark}

\begin{document}

\title[Quadratic rational maps]{On quadratic rational maps with prescribed good reduction}

\author{Clayton Petsche and Brian Stout}

\address{Clayton Petsche; Department of Mathematics; Oregon State University; Corvallis OR 97331 U.S.A.}

\email{petschec@math.oregonstate.edu}

\address{Brian Stout; Ph.D. Program in Mathematics; CUNY Graduate Center; 365 Fifth Avenue; New York, NY 10016-4309 U.S.A.}

\email{bstout@gc.cuny.edu}

\thanks{{\em Date.} August 27, 2012}

\begin{abstract}
Given a number field $K$ and a finite set $S$ of places of $K$, the first main result of this paper shows that the quadratic rational maps $\phi:\PP^1\to\PP^1$ defined over $K$ which have good reduction at all places outside $S$ comprise a Zariski-dense subset of the moduli space $\Mcal_2$ parametrizing all isomorphism classes of quadratic rational maps.  We then consider quadratic rational maps with double unramified fixed-point structure, and our second main result establishes a geometric Shafarevich-type non-Zariski-density result for the set of such maps with good reduction outside $S$.  We also prove a variation of this result for quadratic rational maps with unramified $2$-cycle structure. 
\end{abstract}

\maketitle

%%%%%%%%%%%%%%%%%%%%%%%%%%%%%
%%%%%%%%%%%%%%%%%%%%%%%%%%%%%
%%%%%%%%%%%%%%%%%%%%%%%%%%%%%
%%%%%%%%%%%%%%%%%%%%%%%%%%%%%
%%%%%%%%%%%%%%%%%%%%%%%%%%%%%

\section{Introduction}\label{Introduction}

Let $K$ be a number field, and let $S$ be a finite set of places of $K$ which includes all of the Archimedean places.  In 1963, Shafarevich proved that, up to $K$-isomorphism, there exist only finitely many elliptic curves over $K$ having good reduction at all places $v\not\in S$ (see \cite{silverman:aec} $\S$~IX.6).  He conjectured a generalization of the result to abelian varieties, and this was proved in 1983 by Faltings \cite{MR718935} as part of his proof of Mordell's conjecture.

Motivated by an analogy between elliptic curves and dynamical systems on the projective line, Szpiro and Tucker \cite{MR2435841} have asked whether there is a similar finiteness result for the set $\Rat_d(K)$ of rational maps $\phi:\PP^1\to\PP^1$ of a given degree $d$ defined over $K$.  As they point out, however, if one uses the standard notions of isomorphism and good reduction for rational maps on $\PP^1$, then simple counterexamples preclude a finiteness result of this type.  For example, a rational map defined by a monic integral polynomial has everywhere good reduction, and for each fixed degree $d\geq2$ one can easily find infinite families of such maps which are pairwise non-isomorphic.  

By using a weaker notion of isomorphism defined by separate pre-composition and post-composition actions of $\PGL_2$ on the set of rational maps of degree $d$, and by suitably altering the notion of good reduction, Szpiro-Tucker obtained a finiteness result of Shafarevich type for a certain class of rational maps.  More recently, Petsche \cite{PetscheCSRM} has proved a different Shafarevich-type finiteness theorem along certain families of critically separable rational maps, using a notion of isomorphism defined via $\PGL_2$-conjugation.

In the present paper we consider the following similar but somewhat more geometric question.  Rather than a finiteness statement for the set of isomorphism classes of rational maps of degree $d$ having prescribed good reduction, we ask instead whether or not this set is Zariski-dense in an appropriate moduli space.  We focus here  exclusively on the simplest nontrivial setting of quadratic rational maps (the case $d=2$), and we carry out these investigations in the context of the standard moduli space $\Mcal_2$ of isomorphism classes of such maps.  

Introduced in the complex-analytic setting by Milnor \cite{milnor:quadraticmaps}, and further developed geometrically by Silverman \cite{silverman:modulirationalmaps}, $\Mcal_2$ is an affine surface whose definition arises, via geometric invariant theory, as the quotient 
\begin{equation*}
\Mcal_2=\Rat_2/\PGL_2
\end{equation*}
of the space $\Rat_2$ of all quadratic rational maps $\phi:\PP^1\to\PP^1$, modulo the conjugation action of $\PGL_2$, the automorphism group of $\PP^1$.

The first main result of this paper shows that, under standard definitions of isomorphism and good reduction, not only are there infinitely many non-isomorphic quadratic rational maps over $K$ having good reduction at all places outside $S$, but moreover, these rational maps are ubiquitous enough to comprise a Zariski-dense subset of the moduli space $\Mcal_2$.  The definitions of the required terms are stated precisely in $\S$\ref{QuadRatZarDenseSect}, and this Zariski-density result is stated as Theorem~\ref{ZariskiDenseThm}. 

The remainder of this paper is spent showing that, despite the aforementioned Zariski-density result, if we replace arbitrary quadratic rational maps with objects possessing slightly more dynamical structure, then it is possible to obtain non-Zariski-density results of Shafarevich type for the moduli space $\Mcal_2$.

To explain our motivation in trying to obtain such results, we again consider the analogy with elliptic curves.  It is an interesting fact that Shafarevich's finiteness theorem for $K$-isomorphism classes of elliptic curves may fail, in general, for genus-one curves; see \cite{MR828821} p. 241.  In this setting, the existence of a $K$-rational point acting as the origin for the group law on an elliptic curve has a strong effect on the possible number of $K$-isomorphism classes of such curves.

Motivated by the elliptic curve analogy, we consider quadratic rational maps with double unramified fixed-point structure.  More precisely, we consider the space of triples $\Phi=(\phi,P_1,P_2)$, where $\phi:\PP^1\to\PP^1$ is a quadratic rational map defined over $K$, and where $P_1,P_2\in\PP^1(K)$ are distinct $K$-rational unramified fixed points of $\phi$.  The conjugation action of $\PGL_2$ gives rise to a notion of $K$-isomorphism between two such triples, and we formulate a natural definition of good reduction for such a triple $\Phi=(\phi,P_1,P_2)$ at a non-Archimedean place $v$ of $K$ which, roughly speaking, requires that (up to $K$-isomorphism) the reduction $\tilde\Phi_v=(\tilde\phi_v,\tilde{P}_{1,v},\tilde{P}_{2,v})$ constitutes a quadratic rational map with double unramified fixed-point structure over the residue field $\FF_v$ at $v$.  In $\S$\ref{Rat*} we give the precise definitions of these terms and we prove the second main result of this paper, Theorem~\ref{ZariskiNonDenseThmFixed}, which shows that among all quadratic rational maps with double unramified fixed-point structure, those having good reduction at all places $v$ outside $S$ comprise a non-Zariski-dense subset of the moduli space $\Mcal_2$.  We regard this result as a geometric Safarevich-type theorem for rational maps.

Finally, with very little extra effort we can also establish a variation on the non-Zariski-density result of Theorem~\ref{ZariskiNonDenseThmFixed}, in which maps with unramified $2$-cycle structure take the place of maps with unramified double fixed-point structure.  This result is stated as Theorem~\ref{ZariskiNonDenseThmCycle}.

The first author is supported by NSF grant DMS-0901147.  The second author would like to thank Lucien Szpiro for his generous support under NSF grant DMS-0739346.

%%%%%%%%%%%%%%%%%%%%%%%%%%%%%
%%%%%%%%%%%%%%%%%%%%%%%%%%%%%
%%%%%%%%%%%%%%%%%%%%%%%%%%%%%
%%%%%%%%%%%%%%%%%%%%%%%%%%%%%
%%%%%%%%%%%%%%%%%%%%%%%%%%%%%

\section{Preliminaries}

\subsection{Review of quadratic rational maps}\label{QuadRatMapsSect}  We now fix notation and review basic facts about quadratic rational maps on the projective line; for further details see \cite{MR2316407}, $\S$2.4, $\S$4.3, $\S$4.6.  

An arbitrary quadratic rational map $\phi:\PP^1\to\PP^1$ defined over $\Kbar$ is given in homogeneous coordinates as 
\begin{equation*}
\phi(X:Y) = (A(X,Y):B(X,Y)),
\end{equation*} 
where
\begin{equation*}
\begin{split}
A(X,Y) & =a_0X^2+a_1XY+a_2Y^2 \\
B(X,Y) & =b_0X^2+b_1XY+b_2Y^2
\end{split}
\end{equation*}
are two binary quadratic forms in $\Kbar[X,Y]$ having no common zeros in $\Kbar^2\setminus\{(0,0)\}$.  The requirement that $A(X,Y)$ and $B(X,Y)$ share no common zeros in $\Kbar^2\setminus\{(0,0)\}$ is equivalent to the nonvanishing of the resultant
\begin{equation*} 
\Res(A,B) = \left| \begin{array}{cccc}
a_0 & a_1 & a_2 & 0 \\
0 & a_0 & a_1 & a_2 \\
b_0 & b_1 & b_2 & 0 \\
0 & b_0 & b_1 & b_2 
\end{array} \right|
\end{equation*}
associated to the pair $(A,B)$.  

The group variety of automorphisms of $\PP^1$ is denoted by $\PGL_2$, and each $f\in\PGL_2$ is given in homogeneous coordinates by 
\begin{equation*}
f(X:Y)=(\alpha X+\beta Y:\gamma X+\delta Y)
\end{equation*}
for some nonsingular matrix $\left(\begin{smallmatrix} \alpha&\beta \\ \gamma& \delta  \end{smallmatrix}\right)$ with coefficients in $\Kbar$.  Given a quadratic rational map $\phi:\PP^1\to\PP^1$, we denote by $\phi^f:\PP^1\to\PP^1$ the rational map $\phi^f=f^{-1}\circ\phi\circ f$ defined via conjugation of $\phi$ by $f$.  Explicitly, if we denote by $(A,B):\Kbar^2\to\Kbar^2$ the map defined by $(X,Y)\mapsto(A(X,Y),B(X,Y))$, and if we define binary quadratic forms $C(X,Y)$ and $D(X,Y)$ in $\Kbar[X,Y]$ by the formula
\begin{equation*}
(C,D) = \left(\begin{matrix} \alpha&\beta \\ \gamma& \delta  \end{matrix}\right)^{-1}\circ(A,B)\circ\left(\begin{matrix} \alpha&\beta \\ \gamma& \delta   \end{matrix}\right),
\end{equation*}
then $\phi^f(X:Y) = (C(X,Y):D(X,Y))$.  The formula 
\begin{equation}\label{ResChangeForm}
\Res(C,D) = (\alpha\delta-\beta\gamma)^{2}\Res(A,B),
\end{equation}
which shows the effect of $\GL_2$-conjugation on the resultant, can be verified from direct calulation.

\subsection{Review of the moduli space $\Mcal_2$}\label{ModuliSpaceSect}  The moduli space $\Mcal_2$ parametrizing isomorphism classes of quadratic rational maps was first studied complex analytically by Milnor \cite{milnor:quadraticmaps}, who showed that it is isomorphic to the affine plane $\AA^2$.  A bit later, Silverman \cite{silverman:modulirationalmaps} used geometric invariant theory to construct $\Mcal_2$ (and more generally the moduli space $\Mcal_d$ of rational maps of degree $d$) as a scheme over $\Spec(\ZZ)$, and established the isomorphism $\Mcal_2\simeq\AA^2$ in this more geometric context.  Since that time, variations and generalizations have been studied by Petsche-Szpiro-Tepper \cite{MR2567424}, Levy \cite{MR2741188}, and others.  Further references include \cite{MR2316407} $\S$4.4 and \cite{MR2884382}.

We now review the definition and basic properties of the space $\Mcal_2$.  The first step is to observe that the set of all quadratic rational maps $\phi:\PP^1\to\PP^1$ defined over $\Kbar$ is parametrized by an affine variety which is commonly denoted by $\Rat_2$.  To obtain this variety, note that since the map $\phi$ is unchanged by scaling its coefficients, one may identify $\phi$ with the point $(\abf:\bbf)=(a_0:a_1:a_2:b_0:b_1:b_2)$ of $\PP^5$ defined by its coefficients.  In this way, the space $\Rat_2$ of all quadratic rational maps $\phi:\PP^1\to\PP^1$ is identified with the open affine subvariety $\{(\abf:\bbf)\mid\Res(A,B)\neq0\}$ of $\PP^5$.

Let
\begin{equation}\label{PGLAction}
\begin{split}
\PGL_2\times\Rat_2 & \to \Rat_2 \\
(f,\phi) & \mapsto \phi^f
\end{split}
\end{equation}
be the conjugation action of $\PGL_2$ on $\Rat_2$.  The moduli space $\Mcal_2$ is defined to be the affine variety $\Spec(\Acal^{\PGL_2})$, where $\Acal^{\PGL_2}$ is the subring of $\PGL_2$-invariants in the coordinate ring $\Acal=\Gamma(\Rat_2,\Ocal_{\Rat_2})$.  Using standard facts from geometric invariant theory it can be shown that $\Mcal_2$ is a geometric quotient for the action $(\ref{PGLAction})$, which means roughly that the map
\begin{equation}\label{M2QuotientMap}
\langle\cdot\rangle:\Rat_2\to\Mcal_2
\end{equation}
induced by inclusion $\Acal^{\PGL_2}\subset\Acal$ possesses many of the nice properties one would expect from the quotient map of a group action in the classical sense.  For example, the (geometric) fibers of the map $(\ref{M2QuotientMap})$ are closed, and they are precisely the orbits in $\Rat_2$ with respect to the conjugation action of $\PGL_2$. 

To describe Milnor's isomorphism $\Mcal_2\simeq\AA^2$ in detail, recall that each quadratic rational map $\phi:\PP^1\to\PP^1$ has three (counting with multiplicity) fixed points $\alpha_1,\alpha_2,\alpha_3$ in $\PP^1$, and for each fixed point, the multiplier $\lambda_j$ associated to $\alpha_j$ is the leading coefficient $\lambda_j=\phi'(\alpha_j)$ of the power series expansion of $\phi(z)$ at $z=\alpha_j$ (with respect to some choice of affine coordinate $z$ on $\PP^1$).  A standard calculation shows that the multiplier of a fixed point is invariant under $\PGL_2$-conjugation.  Since the fixed-point set $\Fix(\phi)$ is naturally an unordered triple, we obtain two scalar-valued $\PGL_2$-invariant functions $\sigma_1$ and $\sigma_2$ on $\Rat_2$, defined by the first two symmetric functions $\sigma_1(\phi)=\lambda_1+\lambda_2+\lambda_3$ and $\sigma_2(\phi)=\lambda_1\lambda_2+\lambda_1\lambda_3+\lambda_2\lambda_3$ in the multipliers of the three fixed points of $\phi$.  Milnor's isomorphism $\Mcal_2\simeq\AA^2$ is defined by
\begin{equation}\label{MilnorMap}
\begin{split}
\sigma:\Mcal_2 & \stackrel{\sim}{\to} \AA^2 \\
	\langle\phi\rangle & \mapsto (\sigma_1(\phi),\sigma_2(\phi)).
\end{split}
\end{equation} 
See \cite{milnor:quadraticmaps}.

\subsection{Number-theoretic preliminaries}  We denote by $M_K$ the set of places of the number field $K$.  For each $v\in M_K$, the notation $|\cdot|_v$ refers to any absolute value on $K$ associated to $v$.  If $v$ is non-Archimedean, $\Ocal_v$ is the subring of $v$-integral elements of $K$, $\Mcal_v$ is the unique maximal ideal of $\Ocal_v$, and $\Ocal_v^\times$ is the group of units in $\Ocal_v$.  The notation $x\mapsto\tilde{x}_v$ will denote the reduction map $\Ocal_v\to\FF_v$ onto the residue field $\FF_v=\Ocal_v/\Mcal_v$; we may omit the subscript, writing $\tilde{x}$ instead of $\tilde{x}_v$, if there is no chance for confusion.

The letter $S$ denotes a finite subset of $M_K$ which includes all of the Archimedean places.  The ring of $S$-integers in $K$ is written $\Ocal_S$, and $\Ocal_S^\times$ denotes the group of units in this ring.

Given a quadratic rational map $\phi:\PP^1\to\PP^1$ defined over $K$, and a non-Archimedean place $v$ of $K$, we may scale the coefficients of $A(X,Y)$ and $B(X,Y)$ by a uniformizing parameter at $v$ to obtain all coefficients in $\Ocal_v$, with at least one coefficient in $\Ocal_v^\times$.  Reducing modulo the maximal ideal of $\Ocal_v$, we obtain a reduced rational map $\tilde\phi_v:\PP^1\to\PP^1$ over the residue field $\FF_v$ of $\Ocal_v$, defined by
\begin{equation*}
\tilde\phi_v(X:Y) = (\tilde{A}(X,Y):\tilde{B}(X,Y)).
\end{equation*}
Clearly $\deg(\tilde\phi_v)\leq2$, and it follows from properties of the resultant that $\deg(\tilde\phi_v)=2$ if and only if $\Res(A,B)\in\Ocal_v^\times$.

Given a point $P=(a:b)\in\PP^1(K)$ and a non-Archimedean place $v$ of $K$, we may scale the coordinates $a$ and $b$ by a uniformizing parameter at $v$ to obtain both coordinates in $\Ocal_v$, with at least one of the two in $\Ocal_v^\times$; we obtain a reduced point $\tilde{P}_v=(\tilde{a}_v:\tilde{b}_v)\in\PP^1(\FF_v)$. 

%%%%%%%%%%%%%%%%%%%%%%%%%%%%%
%%%%%%%%%%%%%%%%%%%%%%%%%%%%%
%%%%%%%%%%%%%%%%%%%%%%%%%%%%%
%%%%%%%%%%%%%%%%%%%%%%%%%%%%%
%%%%%%%%%%%%%%%%%%%%%%%%%%%%%

\section{Prescribed good reduction for quadratic rational maps}\label{QuadRatZarDenseSect}

In this section we show that, not only are there infinitely many non-isomorphic quadratic rational maps over $K$ having good reduction at all places outside $S$, but moreover, these rational maps are ubiquitous enough to comprise a Zariski-dense subset of the moduli space $\Mcal_2$.  We first recall the standard definitions of $K$-isomorphism and good reduction.

\begin{dfns}
Two quadratic rational maps $\phi,\psi\in\Rat_2(K)$ are {\em $K$-isomorphic} if $\psi=\phi^f$ for some automorphism $f\in\PGL_2(K)$.  A quadratic rational map $\phi\in\Rat_2(K)$ has {\em good reduction} at a non-Archimedean place $v$ of $K$ if it is $K$-isomorphic to some $\psi\in\Rat_2(K)$ such that $\deg(\tilde\psi_v)=2$.
\end{dfns}

\begin{thm}\label{ZariskiDenseThm}
Let $\Gcal_2(K,S)$ be the set of all points $\langle\phi\rangle$ in $\Mcal_2$ for maps $\phi$ in $\Rat_{2}(K)$ having good reduction at all places $v\in M_K\setminus S$.  Then $\Gcal_2(K,S)$ is Zariski-dense in $\Mcal_2$.
\end{thm}

Our primary proof of this theorem uses the isomorphism $(\ref{MilnorMap})$, as well as a further result of Milnor \cite{milnor:quadraticmaps} on quadratic rational maps in critical-point normal form.  We also give an alternate proof which holds only when the group of $S$-units in $K$ is infinite (thus, this alternate proof only fails to apply when $K$ is either $\QQ$ or a quadratic imaginary extension of $\QQ$, and $S$ consists of the sole Archimedean place of $K$).  While it does not apply in full generality, this secondary proof is sufficiently different from the first proof that it may be of some interest.  It is more self-contained, in that it does not rely on special properties of quadratic rational maps in critical-point normal form, and the ideas behind this secondary proof may find wider applicability toward possible generalizations to the higher degree case.

\begin{proof}[Proof of Theorem~\ref{ZariskiDenseThm}]  
We will consider maps $\phi\in\Rat_2$ given in critical-point normal form
\begin{equation}\label{CPNF}
\phi(X:Y)=(aX^2+bY^2:cX^2+dY^2).
\end{equation}  
The critical points are $(1:0)$ and $(0:1)$, and 
\begin{equation}\label{CPNFRes}
\Res(aX^2+bY^2,cX^2+dY^2)=(ad-bc)^2.
\end{equation}  
Milnor shows in \cite{milnor:quadraticmaps} Corollary C.4 that the two expressions $A=\frac{ad}{ad-bc}$ and $\Sigma=\frac{a^3b+cd^3}{(ad-bc)^2}$ are $\PGL_2$-invariant, and that these quantities are related to the coordinates of the isomorphism $(\ref{MilnorMap})$ by the identities
\begin{equation*}
\begin{split}
\sigma_1 & = 8A-6 \\
\sigma_2 & = 8A^2-20A +4\Sigma+12.
\end{split}
\end{equation*}

Let $\Fcal$ be the image in $\Mcal_2$ of the set of all $\phi\in\Rat_2(\QQ)$ given in critical-point normal form $(\ref{CPNF})$ with $a,b,c,d\in\ZZ$ and $ad-bc=1$.  It follows from $(\ref{CPNFRes})$ that each such map has good reduction at all non-Archimedean places $v\in M_K$; thus $\Fcal\subset\Gcal_2(K,S)$, and we are reduced to showing that $\Fcal$ is Zariski dense in $\Mcal_2$.  

In view of the isomorphism $(\ref{MilnorMap})$, we may identify $\Mcal_2$ with the affine plane $\AA^2$ and we may use $\sigma_1$ and $\sigma_2$ as the two affine coordinates on $\Mcal_2$.  Arguing by contradiction, assume on the contrary that the Zariski closure $\overline\Fcal$ of $\Fcal$ is not all of $\Mcal_2$.  Then $\overline\Fcal$ is a finite union of curves and points in $\Mcal_2$.  By Bezout's theorem, there exists a positive bound $B=B(\overline\Fcal)>0$ such that, if $L$ is any line in $\Mcal_2$, then either $L\subseteq\overline\Fcal$ or $|L\cap\overline\Fcal|\leq B$.  For each $\alpha\in\Kbar$, let $L_\alpha$ be the vertical line $\{\sigma_1=\alpha\}$ in $\Mcal_2$.  Then $\overline\Fcal$ can contain at most finitely many of these lines, call them $L_{\alpha_1},\dots,L_{\alpha_r}$.

We will obtain a contradiction by showing that $\Fcal$ can meet a vertical line $L_\alpha$ at an arbitrarily large number of points, for lines $L_\alpha\not\in\{L_{\alpha_1},\dots,L_{\alpha_r}\}$.  Let $N$ be a positive integer, let $p$ be an arbitrary prime number, and for each $0\leq n\leq N-1$ define
\begin{equation*}
\phi_{n,N}(X:Y)=(p^nX^2+Y^2:(p^{2N}-1)X^2+p^{2N-n}Y^2).
\end{equation*}
Since $(p^n)(p^{2N-n})-(p^{2N}-1)(1)=1$, we have $\phi_{n,N}\in\Fcal$; denote by $\Fcal(N)=\{\langle\phi_{n,N}\rangle\mid 0\leq n\leq N-1\}$.  Using Milnor's calculation of $\sigma_1$ and $\sigma_2$ in terms of $A$ and $\Sigma$, we have
\begin{equation*}
\begin{split}
\sigma_1(\langle\phi_{n,N}\rangle) & = 8p^{2N}-6 \\
\sigma_2(\langle\phi_{n,N}\rangle) & = 8p^{4N}-20p^{2N} +4(p^{3n}+(p^{2N}-1)p^{6N-3n})+12
\end{split}
\end{equation*}
This shows that $\Fcal(N)$ is contained in the line $L_{8p^{2N}-6}$.  Further, note that for fixed $N$, the numbers $p^{3n}+(p^{2N}-1)p^{6N-3n}$ are distinct as $n$ ranges from $0\leq n\leq N-1$ (for example, because the $p$-adic absolute value of $p^{3n}+(p^{2N}-1)p^{6N-3n}$ is $p^{-3n}$).  Therefore, the $\sigma_2$-coordinates of the $N$ points $\langle\phi_{n,N}\rangle$ are distinct for $0\leq n\leq N-1$, whereby $|\Fcal(N)|=N$.  We have shown that $\Fcal$ meets the line $L_{8p^{2N}-6}$ in at least $N$ points; taking $N$ large enough produces a contradiction, since there are only finitely many $N$ for which $N\leq B$ or  $L_{8p^{2N}-6}\in\{L_{\alpha_1},\dots,L_{\alpha_r}\}$.
\end{proof}

\begin{proof}[Alternate proof of Theorem~\ref{ZariskiDenseThm}]  (This proof only holds under the additional assumption that the $S$-unit group $\Ocal_S^\times$ of $K$ is infinite).

We will consider maps $\phi\in\Rat_2$ given in fixed-point normal form
\begin{equation}\label{FPNF}
\phi(X:Y)=(X^2+\lambda_1XY:\lambda_2XY+Y^2).
\end{equation}  
For rational maps in this form, the fixed points, their multipliers, and the resultant are particularly easy to calculate.  The fixed points of $\phi$ are $(0:1)$, $(1:0)$, and $(1-\lambda_1:1-\lambda_2)$, with multipliers $\lambda_1, \lambda_2$, and $\lambda_3=\frac{2-\lambda_1-\lambda_2}{1-\lambda_1\lambda_2}$, respectively (see \cite{MR2316407} $\S$4.6), and
\begin{equation*}
\Res(X^2+\lambda_1XY,\lambda_2XY+Y^2)=1-\lambda_1\lambda_2.
\end{equation*}

For each pair of nonzero elements $\alpha,\beta\in \Kbar^\times$, define $\phi_{\alpha,\beta}\in\Rat_2$ to be the map given in fixed-point normal form $(\ref{FPNF})$ with $\lambda_1=\alpha$ and $\lambda_2=\frac{1-\beta}{\alpha}$.  We obtain a map
\begin{equation*}
\begin{split}
u: \GG_m\times\GG_m & \rightarrow\Mcal_2 \\
u(\alpha,\beta) & = \langle\phi_{\alpha,\beta}\rangle.
\end{split}
\end{equation*}

We first show that $u$ is dominant.  For each $\alpha\in \Kbar^\times$, define $u_\alpha:\GG_m\rightarrow\Mcal_2$ by $u_\alpha(z)=u(\alpha,z)$, and let $Z_\alpha=\overline{u_\alpha(\GG_m)}$ be the Zariski-closure of the image of $u_\alpha$.  In view of the isomorphism $(\ref{MilnorMap})$, we may identify $\Mcal_2$ with the affine plane $\AA^2$ and we may use $\sigma_1$ and $\sigma_2$ as the two affine coordinates on $\Mcal_2$.  Direct calculations show that $u_1(z)=(3-z,3-2z)$ and therefore $Z_1$ is the line $2\sigma_1-\sigma_2=3$ in $\Mcal_2$.  Similarly, $u_{-1}(z)=(-3+z+\frac{4}{z},7-2(z+\frac{4}{z}))$ and therefore $Z_{-1}$ is the line $2\sigma_1+\sigma_2=1$ in $\Mcal_2$.  Now let $Z=\overline{u(\Ocal_S^\times \times\Ocal_S^\times)}$ be the Zariski-closure of the image of $u$. Then since the torus $\GG_m\times\GG_m$ is irreducible, $Z$ is irreducible, and since $Z$ contains the two distinct lines $Z_1$ and $Z_{-1}$, $Z$ must have dimension $2$.  Therefore $Z=\Mcal_2$ and $u$ is dominant. 

Define $\Vcal=u(\Ocal_S^\times\times\Ocal_S^\times)$ to be the image in $\Mcal_2$ under $u$ of the set of all $\phi_{\alpha,\beta}\in\Rat_2(K)$ for which both $\alpha$ and $\beta$ are in the $S$-unit group $\Ocal_S^\times$.  The calculation $\Res(X^2+\alpha XY,(\frac{1-\beta}{\alpha})XY+Y^2)=\beta$ shows that each such map has good reduction at all places $v\in M_K\setminus S$, and therefore $\Vcal\subset\Gcal_2(K,S)$.  To complete the proof of the theorem, we only need to show that $\Vcal$ is Zariski-dense in $\Mcal_2$.

Since $\Ocal_S^\times$ is infinite, the subgroup $\Ocal_S^\times \times\Ocal_S^\times$ is Zariski-dense in $\GG_m\times\GG_m$.  Therefore
\begin{equation}\label{ZarClosure}
u(\GG_m\times\GG_m)=u(\overline{\Ocal_S^\times \times\Ocal_S^\times})\subseteq\overline{u(\Ocal_S^\times \times\Ocal_S^\times)}=\overline\Vcal;
\end{equation}
here we have used the fact, which is true of all continuous maps on topological spaces, including morphisms of algebraic varieties, that $f(\overline{X})\subseteq \overline{f(X)}$.  Taking the Zariski-closure of both sides of $(\ref{ZarClosure})$ we obtain $\overline{u(\GG_m\times\GG_m)}\subseteq\overline\Vcal$.  Since $u$ is dominant, we have $\overline{u(\GG_m\times\GG_m)}=\Mcal_2$, and therefore $\overline\Vcal=\Mcal_2$, completing the proof.
\end{proof}

%%%%%%%%%%%%%%%%%%%%%%%%%%%%%
%%%%%%%%%%%%%%%%%%%%%%%%%%%%%
%%%%%%%%%%%%%%%%%%%%%%%%%%%%%
%%%%%%%%%%%%%%%%%%%%%%%%%%%%%
%%%%%%%%%%%%%%%%%%%%%%%%%%%%%

\section{Prescribed good reduction for quadratic rational maps with double unramified fixed-point structure}\label{Rat*}

One of the goals of the paper is to emphasize that dynamical analogues of theorems for elliptic curves over number fields may fail because general rational maps lack the richer structure of elliptic curves.  As mentioned in the introduction, when one replaces elliptic curve in the statement of Shafarevich's theorem with genus-one curve, the theorem becomes false.  In that setting, the extra structure provided by a marked $K$-rational point acting as the origin for the elliptic curve has a dramatic influence on the set of $K$-isomorphism classes of such objects.  With the elliptic curve analogy in mind, in this section we consider rational maps equipped with some additional structure arising from fixed points.

\begin{dfn}
Let $\Rat_{2,2}^{\uf}(K)$ be the set of all triples of the form $\Phi=(\phi,P_1,P_2)$, where $\phi:\PP^1\to\PP^1$ is a quadratic rational map defined over $K$, and where $P_1,P_2\in\PP^1(K)$ are distinct $K$-rational unramified fixed points of $\phi$.  We call such a triple $\Phi$ a \emph{quadratic rational map with double unramified fixed-point structure over $K$}; or, when the context is clear, for brevity we may refer to $\Phi$ simply as a map. 
\end{dfn}

We first remark that the set of quadratic rational maps which fail to have three distinct unramified fixed points forms a proper Zariski-closed subset of $\Mcal_2$.  Further, given a map $\phi\in\Rat_2(K)$, the three fixed points of $\phi$ are $K'$-rational for some extension $K'/K$ of degree at most six.  So from the point of view studying geometric questions concerning generic rational maps, the extra conditions required of a quadratic rational map with double unramified fixed-point structure over $K$ are not terribly restrictive.  (Later we will explain our choice of double, rather than single or triple, unramified fixed-point structure.)  

Given a map $\Phi=(\phi,P_1,P_2)$ in $\Rat_{2,2}^{\uf}(K)$, and an automorphism $f\in\PGL_2(K)$, observe that $f^{-1}(P_1)$ and $f^{-1}(P_2)$ are distinct unramified fixed points of $\phi^f=f^{-1}\circ\phi\circ f$; we may therefore define $\Phi^f\in\Rat_{2,2}^{\uf}(K)$, the conjugate of $\Phi$ with respect to $f$, by
\begin{equation*}
\Phi^f=(\phi^f,f^{-1}(P_1),f^{-1}(P_2)).
\end{equation*}  
This notion of $\PGL_2(K)$-conjugation on $\Rat_{2,2}^{\uf}(K)$ gives rise to the following definitions.

\begin{dfns}
Two maps $\Phi$ and $\Psi$ in $\Rat_{2,2}^{\uf}(K)$ are {\em $K$-isomorphic} if $\Psi=\Phi^f$ for some automorphism $f\in\PGL_2(K)$.  A map $\Phi$ in $\Rat_{2,2}^{\uf}(K)$ has {\em good reduction} at a non-Archimedean place $v$ of $K$ if it is $K$-isomorphic to some $\Psi=(\psi,Q_1,Q_2)$ in $\Rat_{2,2}^{\uf}(K)$ such that $\deg(\tilde\psi_v)=2$ and such that $\tilde{Q}_1$ and $\tilde{Q}_2$ are distinct unramified fixed points of $\tilde\psi_v$.
\end{dfns}

We emphasize that this notion of good reduction for a map $\Phi=(\phi,P_1,P_2)$ in $\Rat_{2,2}^{\uf}(K)$ is slightly stronger than the standard definition of good reduction for its underlying rational map $\phi\in\Rat_2(K)$; a natural additional condition has been added to ensure that reduction modulo the maximal ideal of $\Ocal_v$ preserves the double unramified fixed-point structure of the triple $\tilde\Phi=(\tilde\phi_v,\tilde{P}_{1,v},\tilde{P}_{2,v})$ over the residue field $\FF_v$.

Abusing notation slightly, for each $\Phi=(\phi,P_1,P_2)$ in $\Rat_{2,2}^{\uf}(K)$, define $\langle\Phi\rangle=\langle\phi\rangle$.  Thus, one may view $\langle\cdot\rangle:\Rat_{2,2}^{\uf}(K)\to\Mcal_2(K)$ as the map which forgets the fixed-point structure of $\Phi$ and preserves only the $\PGL_2$-conjugacy class $\langle\phi\rangle$ of its underlying rational map.

The main theorem of this section is the following, which shows that the set of all $\Phi$ in $\Rat_{2,2}^{\uf}(K)$ having good reduction outside $S$ comprises a non-Zariski-dense subset of the moduli space $\Mcal_2$.

\begin{thm}\label{ZariskiNonDenseThmFixed}
Let $\Gcal_{2,2}^{\uf}(K,S)$ be the set of all points $\langle\Phi\rangle$ in $\Mcal_2$ for maps $\Phi$ in $\Rat_{2,2}^{\uf}(K)$ having good reduction at all places $v\in M_K\setminus S$.  Then $\Gcal_{2,2}^\uf(K,S)$ is not Zariski-dense in $\Mcal_2$.
\end{thm}

We need three preliminary propositions before we can give the proof of Theorem~\ref{ZariskiNonDenseThmFixed}.  The first states that maps in $\Rat_{2,2}^{\uf}(K)$ having good reduction at a non-Archimedean place $v$ of $K$ can be represented (up to $K$-isomorphism) in a certain simple form. 

\begin{prop}\label{GoodRedLocalProp}
Suppose that $\Phi\in\Rat_{2,2}^{\uf}(K)$ has good reduction at a non-Archimedean place $v$ of $K$.  Then $\Phi$ is $K$-isomorphic to $\Psi=(\psi,(1:0),(0:1))$ for some quadratic rational map $\psi:\PP^1\to\PP^1$ given by
\begin{equation}\label{GoodRedForm}
\psi(X:Y)=(X^2+aXY:bXY+cY^2)
\end{equation}
for $a, b, c\in\Ocal_v^\times$ such that 
\begin{equation*}
\Res(X^2+aXY,bXY+cY^2))=c(c-ab)\in\Ocal_v^\times.
\end{equation*} 
\end{prop}
\begin{proof}
According to the definition of good reduction, possibly replacing $\Phi$ with some member of its $K$-isomorphism class, we may assume without loss of generality that $\Phi=(\phi,P_1,P_2)$, where $\deg(\tilde\phi_v)=2$ and where $\tilde{P}_1$ and $\tilde{P}_2$ are distinct unramified fixed points of the reduced map $\tilde\phi_v$.  

The fact that $\deg(\tilde\phi_v)=2$ means that we may write $\phi(X,Y)=(A(X,Y):B(X,Y))$ for forms $A(X,Y),B(X,Y)$ in $\Ocal_v[X,Y]$ with $\Res(A,B)\in\Ocal_v^\times$.  Set $P_1=(\alpha_1:\beta_1)$ for $\alpha_1,\beta_1\in\Ocal_v$ and at least one of the two $\alpha_1,\beta_1$ in the unit group $\Ocal_v^\times$, and set $P_2=(\alpha_2:\beta_2)$ subject to the same requirements.  Since $\tilde{P}_{1}\neq\tilde{P}_{2}$ in $\PP^1(\FF_v)$, we have $\tilde{\alpha}_{2}\tilde{\beta}_{1}-\tilde{\alpha}_{1}\tilde{\beta}_{2}\neq0$ in $\FF_v$.  In other words $\alpha_2\beta_1-\alpha_1\beta_2\in\Ocal_v^\times$, and therefore the matrix $\left(\begin{smallmatrix} \alpha_1&\alpha_2 \\ \beta_1& \beta_2  \end{smallmatrix}\right)$ is an element of $\GL_2(\Ocal_v)$.

Define $\psi=\phi^f$, where $f\in\PGL_2(K)$ is given by $f(X:Y)=(\alpha_1X+\alpha_2Y:\beta_1X+\beta_2Y)$.  This means that $\psi(X:Y)=(C(X,Y):D(X,Y))$ where the forms $C(X,Y)$ and $D(X,Y)$ are defined by
\begin{equation*}
(C,D) = \left(\begin{matrix} \alpha_1&\alpha_2 \\ \beta_1& \beta_2  \end{matrix}\right)^{-1}\circ(A,B)\circ\left(\begin{matrix} \alpha_1&\alpha_2 \\ \beta_1& \beta_2   \end{matrix}\right),
\end{equation*}
Since $\left(\begin{smallmatrix} \alpha_1&\alpha_2 \\ \beta_1& \beta_2  \end{smallmatrix}\right)\in\GL_2(\Ocal_v)$, the formula $(\ref{ResChangeForm})$ shows that the forms $C(X,Y)$ and $D(X,Y)$ have coefficients in $\Ocal_v$ and resultant $\Res(C,D)$ in $\Ocal_v^\times$.  In particular, $\deg(\tilde{\psi}_v)=2$.

Since $f(1:0)=P_1$ and $f(0:1)=P_2$, it follows that $(1:0)$ and $(0:1)$ are fixed points of $\psi$, and therefore we have $C(X,Y)=c_0X^2+c_1XY$ and $D(X,Y)=d_1XY+d_2Y^2$ for elements $c_0,c_1,d_1,d_2\in\Ocal_v$, with $\Res(C,D)=c_0d_2(c_0d_2-c_1d_1)\in\Ocal_v^\times$.  This immediately forces $c_0,d_2\in\Ocal_v^\times$, since otherwise $c_0$ or $d_2$ would be an element of the maximal ideal of $\Ocal_v$, making $c_0d_2(c_0d_2-c_1d_1)\in\Ocal_v^\times$ impossible.  

Since $\left(\begin{smallmatrix} \alpha_1&\alpha_2 \\ \beta_1& \beta_2  \end{smallmatrix}\right)\in\GL_2(\Ocal_v)$, the automorphism $f$ reduces to an automorphism $\tilde{f}\in\PGL_2(\FF_v)$, and $\tilde{\psi}_v=\tilde{\phi}_v^{\tilde{f}}$.  Since $\tilde{P}_1$ and $\tilde{P}_2$ are unramified fixed points of $\tilde\phi_v$, it follows that $(\tilde{1}:\tilde{0})$ and $(\tilde{0}:\tilde{1})$ are unramified fixed points of $\tilde{\psi}_v$.  Standard calculations show that, since $(\tilde{1}:\tilde{0})$ is an unramified point of $\tilde\psi_v$, $\tilde{d}_1$ is nonzero in $\FF_v$, and since $(\tilde{0}:\tilde{1})$ is an unramified point of $\tilde\psi_v$, $\tilde{c}_1$ is nonzero in $\FF_v$.  Consequently, both $c_1$ and $d_1$ are in $\Ocal_v^\times$.

Finally, setting $a=\frac{c_1}{c_0}$, $b=\frac{d_1}{c_0}$, and $c=\frac{d_2}{c_0}$, we obtain a representation for the map $\psi$ in the desired form $(\ref{GoodRedForm})$, with $a,b,c\in\Ocal_v^\times$ and $c(c-ab)=\frac{d_2}{c_0}(\frac{d_2}{c_0}-\frac{c_1}{c_0}\frac{d_1}{c_0})=c_0^{-3}d_2(c_0d_2-c_1d_1)\in\Ocal_v^\times$. 
\end{proof}

Given a map $\Phi$ in $\Rat_{2,2}^{\uf}(K)$ having good reduction at all places outside $S$, Proposition~\ref{GoodRedLocalProp} shows that for each $v\in M_K\setminus S$, $\Phi$ is $K$-isomorphic to some map $\Psi$ possessing a particularly simple form which realizes this good reduction at $v$.  A priori, the map $\Psi$ may vary from place to place, but the following proposition shows that, if $\Ocal_S$ is a principal ideal domain, then a global map $\Psi$ can be found satisfying the conclusion of Proposition~\ref{GoodRedLocalProp} at every place $v\in M_K\setminus S$.

\begin{prop}\label{GoodRedGlobalProp}
Assume that $\Ocal_S$ is a principal ideal domain.  Suppose that $\Phi\in\Rat_{2,2}^{\uf}(K)$ has good reduction at all places $v\in M_K\setminus S$.  Then $\Phi$ is $K$-isomorphic to $\Psi=(\psi,(1:0),(0:1))$ for some quadratic rational map $\psi:\PP^1\to\PP^1$ given by
\begin{equation*}
\psi(X:Y)=(X^2+aXY:bXY+cY^2)
\end{equation*}
for $a, b, c\in\Ocal_S^\times$ such that 
\begin{equation*}
\Res(X^2+aXY,bXY+cY^2))=c(c-ab)\in\Ocal_S^\times.
\end{equation*}
\end{prop}

\begin{proof}
Replacing $\Phi=(\phi,P_1,P_2)$ with its conjugate by a suitable automorphism in $\PGL_2(K)$ which takes $(1:0)$ to $P_1$ and $(0:1)$ to $P_2$, we may assume without loss of generality that 
$\Phi=(\phi,(1:0),(0:1))$ for some map $\phi$ given by
\begin{equation*}
\phi(X:Y)=(X^2+a_0XY:b_0XY+c_0Y^2)
\end{equation*}
for coefficients $a_0,b_0,c_0\in K$.

Fix a place $v\in M_K\setminus S$.  Since $\Phi$ has good reduction at $v$, it follows from Proposition~\ref{GoodRedLocalProp} that $\Phi$ is $K$-isomorphic to some map $\Psi_v=(\psi_v,(1:0),(0:1))$, where 
\begin{equation}\label{phiv1}
\psi_v(X:Y)=(X^2+a_vXY:b_vXY+c_vY^2)
\end{equation}
for $a_v, b_v, c_v\in\Ocal_v^\times$ such that
\begin{equation*}
\Res(X^2+a_vXY,b_vXY+c_vY^2))=c_v(c_v-a_vb_v)\in\Ocal_v^\times.
\end{equation*} 
Let $f_v\in\PGL_2(K)$ be the automorphism such that $\Psi_v=\Phi^{f_v}$.  Since both $\psi_v$ and $\phi$ fix the points $(1:0)$ and $(0:1)$, the automorphism $f_v$ must fix these points as well, so we may write $f_v(X,Y)=(\alpha_v X:Y)$ for some $\alpha_v\in K^\times$.  Conjugating $\phi$ by $f_v$ we obtain
\begin{equation}\label{phiv2}
\phi^{f_v}(X:Y)=(X^2+\alpha_v^{-1}a_0XY:b_0XY+\alpha_v^{-1}c_0Y^2).
\end{equation}
Since $\psi_v=\phi^{f_v}$, comparing $(\ref{phiv1})$ and $(\ref{phiv2})$ we obtain the three identities
\begin{equation*}
\begin{split}
a_v & = \alpha_v^{-1}a_0 \\ 
b_v & = b_0 \\
c_v & = \alpha_v^{-1}c_0.
\end{split}
\end{equation*}

Since $\mathcal{O}_S$ is a principal ideal domain, there exists $\alpha\in K^\times$ such that $|\alpha|_v=|\alpha_v|_v$ for each $v\in M_K\setminus S$.  Define $\Psi=(\psi,(1:0),(0:1))$ for
\begin{equation*}
\psi(X:Y)=(X^2+aXY:bXY+cY^2)
\end{equation*}
with coefficients given by
\begin{equation*}
\begin{split}
a & = \alpha^{-1}a_0 \\ 
b & = b_0 \\
c & = \alpha^{-1}c_0.
\end{split}
\end{equation*}
Then $\Psi=\Phi^{f}$ for the automorphism $f\in\PGL_2(K)$ defined by $f(X:Y)=(\alpha X:Y)$.  Furthermore, for each place $v\in M_K\setminus S$ we have $|a|_v=|\alpha^{-1}a_0|_v=|\alpha_v^{-1}a_0|_v=|a_v|_v=1$, whereby $a\in\Ocal_v^\times$; similar calculations show that $b$, $c$, and $c(c-ab)$ are all elements of the unit group  $\Ocal_v^\times$ as well.  Since these elements are in $\Ocal_v^\times$ for all $v\in M_K\setminus S$, we have $a,b, c, c(c-ab)\in\Ocal_S^\times$.
\end{proof}

\begin{lem}\label{ZarClosedLem}
Given $u\in\Kbar$, define $V_u$ to be the set of all $\langle\phi\rangle$ in $\Mcal_2$ for $\phi\in\Rat_2$ of the form 
\begin{equation}\label{NormalForm}
\phi(X:Y)=(X^2+aXY:bXY+cY^2)
\end{equation}
with $\frac{ab}{c}=u$.  Then $V_u$ is a proper Zariski-closed subset of $\Mcal_2$.
\end{lem}

\begin{proof}
Define $W_u$ to be the set of maps $\phi\in\Rat_2$ of the form $(\ref{NormalForm})$ with $\frac{ab}{c}=u$.  In the notation of $\S$\ref{QuadRatMapsSect} and $\S$\ref{ModuliSpaceSect}, $W_u$ is the intersection of the three hypersurfaces $\{a_2=0\}$, $\{b_0=0\}$, and $\{a_1b_1=ua_0b_2\}$ in $\Rat_2$, and thus $W_u$ is Zariski-closed in $\Rat_2$.  Since $V_u$ is the image of $W_u$ under the closed quotient map $\langle\cdot\rangle:\Rat_2\to\Mcal_2$, it follows that $V_u$ is Zariski-closed in $\Mcal_2$.

It remains to show that $V_u$ is a proper subset of $\Mcal_2$; we will prove the stronger statement that $V_u$ and $V_{u'}$ are disjoint whenever $u\neq u'$.  For if $V_u\cap V_{u'}$ is nonempty, then there exist maps $\phi\in W_{u}$ and $\phi'\in W_{u'}$ which are $\PGL_2$-conjugate to one another; say $\phi'=\phi^f$ for $f\in\PGL_2$.  In the obvious notation we therefore have $\frac{ab}{c}=u$ and $\frac{a'b'}{c'}=u'$.  Since both $\phi$ and $\phi'$ fix both $(0:1)$ and $(1:0)$, the automorphism $f$ must fix these points as well, whereby $f(X:Y)=(\alpha X:Y)$ for some $\alpha\in \Kbar$.  But for maps of the form $(\ref{NormalForm})$, the quantity $\frac{ab}{c}$ is invariant under the action of automorphisms of this form, because conjugating $\phi$ by $f$ we have 
\begin{equation*}
\phi^f(X:Y)=(X^2+\alpha^{-1}aXY:bXY+\alpha^{-1}cY^2)
\end{equation*}
and $\frac{(\alpha^{-1}a)b}{\alpha^{-1} c}=\frac{ab}{c}$.  It follows that $u=u'$.  This completes the verification that $V_u\cap V_{u'}=\emptyset$ whenever $u\neq u'$, and therefore each $V_u$ is a proper subset of $\Mcal_2$. 
\end{proof}

\begin{proof}[Proof of Theorem~\ref{ZariskiNonDenseThmFixed}]
As englarging $S$ proves a stronger statement, first let us increase the size of $S$ so that $\mathcal{O}_S$ is a principal ideal domain.

Each point in $\Gcal_{2,2}^{\uf}(K,S)$ is $\langle\Phi\rangle$ for some $\Phi\in\Rat_{2,2}^\uf(K)$ with good reduction at all places $v\in M_K\setminus S$, and according to Proposition~\ref{GoodRedGlobalProp} we may assume without loss of generality that each such map takes the form $\Phi=(\phi,(1:0),(0:1))$ where
\begin{equation*}
\phi(X:Y)=(X^2+aXY:bXY+cY^2)
\end{equation*}
for $a, b, c, c(c-ab)\in\Ocal_S^\times$.  It follows that $(\frac{c-ab}{c}, \frac{ab}{c})$ is a solution in $(\Ocal_S^\times)^2$ to the unit equation $x+y=1$.  Since there are only finitely many such solutions (\cite{bombierigubler} $\S$5.1), there exists a finite list of units $u_1,\dots,u_r\in\Ocal_S^\times$ such that 
\begin{equation}\label{ShowsNotZD}
\Gcal_{2,2}^{\uf}(K,S)\subseteq V_{u_1}\cup\dots\cup V_{u_r},
\end{equation}  
where $V_u$ denotes the the set of all $\langle\phi\rangle$ in $\Mcal_2$ for $\phi\in\Rat_2$ of the form 
$\phi(X:Y)=(X^2+aXY:bXY+cY^2)$ with $\frac{ab}{c}=u$.  

By Lemma~\ref{ZarClosedLem}, each $V_u$ is a proper Zariski-closed subset of $\Mcal_2$, and therefore $(\ref{ShowsNotZD})$ shows that $\Gcal_{2,2}^{\uf}(K,S)$ is not Zariski-dense in $\Mcal_2$.
\end{proof}

\begin{rem}
Since a generic quadratic rational map in $\Rat_2$ has three distinct unramified fixed points over $\Kbar$, it would be reasonable to ask why we consider rational maps with double (rather than single or triple) unramified fixed-point structure.

First, Theorem~\ref{ZariskiNonDenseThmFixed} would be false in general if double unramified fixed-point structure were replaced by {\em single} unramified fixed-point structure, and a counterexample is given by the same family occuring in the second proof of Theorem~\ref{ZariskiDenseThm}.  Recall that $\Vcal$ is the set of all $\langle\phi\rangle$ in $\Mcal_2$ for rational maps $\phi\in\Rat_2(K)$ taking the form $\phi(X:Y)=(X^2+\alpha XY:(\frac{1-\beta}{\alpha})XY+Y^2)$ for $S$-units $\alpha,\beta\in\Ocal_S^\times$.  At all places $v\in M_K\setminus S$, the point $(0:1)$ reduces to an unramified fixed point of $\tilde\phi_v$, but assuming that $\Ocal_S^\times$ is infinite, $\Vcal$ is Zariski-dense in $\Mcal_2$.

One might define the space $\Rat_{3,2}^\uf(K)$ of quadratic rational maps with {\em triple} unramified fixed-point structure over $K$; but the non-Zariski-density of the image in $\Mcal_2$ of the set of all such maps having prescribed good reduction would follow at once from Theorem~\ref{ZariskiNonDenseThmFixed}.  Indeed, the map $\Rat_{3,2}^\uf(K)\to\Mcal_2(K)$ factors through the map $\Rat_{3,2}^\uf(K)\to\Rat_{2,2}^\uf(K)$ which remembers the first two fixed points and forgets the third.  In an obvious extension of the notation of Theorem~\ref{ZariskiNonDenseThmFixed}, we therefore have $\Gcal_{3,2}^\uf(K,S)\subseteq\Gcal_{2,2}^\uf(K,S)$.
\end{rem}

\begin{rem}
Using geometric invariant theory, one might define a new moduli space $\Mcal_{2,2}^{\uf}=\Rat_{2,2}^\uf/\PGL_2$ as the quotient of the space $\Rat_{2,2}^\uf$ of quadratic rational maps with double unramified fixed-point structure modulo the conjugation action of $\PGL_2$.  This space would be of interest on its own merits, but for our purposes there would be little to gain in such a construction, because a non-Zariski-density result in $\Mcal_{2,2}^{\uf}$ for maps with prescribed good reduction would follow trivially from Theorem~\ref{ZariskiNonDenseThmFixed}, using the dominant finite-to-one map $\Mcal_{2,2}^{\uf}\to\Mcal_{2}$ obtained from forgetting the double unramified fixed-point structure.
\end{rem}

%%%%%%%%%%%%%%%%%%%%%%%%%%%%%
%%%%%%%%%%%%%%%%%%%%%%%%%%%%%
%%%%%%%%%%%%%%%%%%%%%%%%%%%%%
%%%%%%%%%%%%%%%%%%%%%%%%%%%%%
%%%%%%%%%%%%%%%%%%%%%%%%%%%%%

\section{Prescribed good reduction for quadratic rational maps with unramified $2$-cycle structure}

With nearly the same proof, a variation on Theorem~\ref{ZariskiNonDenseThmFixed} can be established in which $2$-cycle structure is used in place of double fixed-point structure.

\begin{dfn}
Let $\Rat_{2,2}^{\uc}(K)$ be the set of all triples of the form $\Phi=(\phi,P_1,P_2)$, where $\phi:\PP^1\to\PP^1$ is a quadratic rational map defined over $K$, and where $P_1,P_2\in\PP^1(K)$ are distinct $K$-rational points which are not ramified points of $\phi$ and for which $\phi(P_1)=P_2$ and $\phi(P_2)=P_1$.  We call such a triple $\Phi$ a \emph{quadratic rational map with double unramified $2$-cycle structure over $K$}. 
\end{dfn}

\begin{dfns}
Two maps $\Phi$ and $\Psi$ in $\Rat_{2,2}^{\uc}(K)$ are {\em $K$-isomorphic} if $\Psi=\Phi^f$ for some automorphism $f\in\PGL_2(K)$, where 
\begin{equation*}
\Phi^f=(\phi^f,f^{-1}(P_1),f^{-1}(P_2)).
\end{equation*}  
A map $\Phi$ in $\Rat_{2,2}^{\uc}(K)$ has {\em good reduction} at a non-Archimedean place $v$ of $K$ if it is $K$-isomorphic to some $\Psi=(\psi,Q_1,Q_2)$ in $\Rat_{2,2}^{\uc}(K)$ such that $\deg(\tilde\psi_v)=2$ and such that $\tilde{Q}_1$ and $\tilde{Q}_2$ are distinct points in $\PP^1(\FF_v)$ which are not ramified points of $\tilde\psi_v$ and for which $\tilde\psi_v(\tilde{Q}_1)=\tilde{Q}_2$ and $\tilde\psi_v(\tilde{Q}_2)=\tilde{Q}_1$.  
\end{dfns}

For each $\Phi=(\phi,P_1,P_2)$ in $\Rat_{2,2}^{\uc}(K)$, define $\langle\Phi\rangle=\langle\phi\rangle$.  Thus, as before, the map $\langle\cdot\rangle:\Rat_{2,2}^{\uc}(K)\to\Mcal_2(K)$ forgets the $2$-cycle structure of $\Phi$ and preserves only the $\PGL_2$-conjugacy class $\langle\phi\rangle$ of its underlying rational map.

\begin{thm}\label{ZariskiNonDenseThmCycle}
Let $\Gcal_{2,2}^{\uc}(K,S)$ be the set of all points $\langle\Phi\rangle$ in $\Mcal_2$ for maps $\Phi$ in $\Rat_{2,2}^{\uf}(K)$ having good reduction at all places $v\in M_K\setminus S$.  Then $\Gcal_{2,2}^\uf(K,S)$ is not Zariski-dense in $\Mcal_2$.
\end{thm}

\begin{proof}
This proof follows precisely the same strategy as that of Theorem~\ref{ZariskiNonDenseThmFixed}, and so we only give a sketch to highlight where this proof differs from the previous one.  

Again, without loss of generality we may enlarge $S$ so that $\mathcal{O}_S$ is a principal ideal domain.  Each point in $\Gcal_{2,2}^{\uc}(K,S)$ is $\langle\Phi\rangle$ for some $\Phi\in\Rat_{2,2}^\uc(K)$ with good reduction at all places $v\in M_K\setminus S$, and in a similar fashion as in Proposition~\ref{GoodRedGlobalProp}, it may be shown that each such map (up to $K$-isomorphism) takes the form $\Phi=(\phi,(1:0),(0:1))$ where
\begin{equation*}
\phi(X:Y)=(aXY+bY^2:X^2+cXY)
\end{equation*}
for $a, b, c, b(b-ac)\in\Ocal_S^\times$.  It follows that $(\frac{b-ac}{b}, \frac{ac}{b})$ is a solution in $(\Ocal_S^\times)^2$ to the unit equation $x+y=1$.  Since there are only finitely many such solutions (\cite{bombierigubler} $\S$5.1), there exists a finite list of units $u_1,\dots,u_r\in\Ocal_S^\times$ such that 
\begin{equation*}
\Gcal_{2,2}^{\uf}(K,S)\subseteq Z_{u_1}\cup\dots\cup Z_{u_r},
\end{equation*}  
where $Z_u$ denotes the the set of all $\langle\phi\rangle$ in $\Mcal_2$ for $\phi\in\Rat_2$ of the form 
$\phi(X:Y)=(aXY+bY^2:X^2+cXY)$ with $\frac{ac}{b}=u$, and each $Z_u$ is a proper Zariski-closed subset of $\Mcal_2$.
\end{proof}

%%%%%%%%%%%%%%%%%%%%%%%%%%%%%
%%%%%%%%%%%%%%%%%%%%%%%%%%%%%
%%%%%%%%%%%%%%%%%%%%%%%%%%%%%
%%%%%%%%%%%%%%%%%%%%%%%%%%%%%
%%%%%%%%%%%%%%%%%%%%%%%%%%%%%

\medskip

\bibliographystyle{acm}

\medskip

\end{document}